\documentclass[oneside,english]{amsart}
\usepackage[T1]{fontenc}
\usepackage[latin9]{inputenc}
\usepackage{amstext}
\usepackage{amsthm}
\usepackage{amssymb}
\usepackage{wasysym}
\usepackage{esint}

\makeatletter
\numberwithin{equation}{section}
\numberwithin{figure}{section}
\theoremstyle{plain}
\newtheorem{thm}{\protect\theoremname}
\theoremstyle{definition}
\newtheorem{defn}[thm]{\protect\definitionname}
\theoremstyle{plain}
\newtheorem{lem}[thm]{\protect\lemmaname}
\theoremstyle{remark}
\newtheorem{rem}[thm]{\protect\remarkname}

\makeatother

\usepackage{babel}
\providecommand{\definitionname}{Definition}
\providecommand{\lemmaname}{Lemma}
\providecommand{\remarkname}{Remark}
\providecommand{\theoremname}{Theorem}

\begin{document}
\title{$\sigma_{2}$ yamabe problem on conic 4-spheres }
\author{Hao Fang}
\address{14 MacLean Hall, University of Iowa, Iowa City, IA, 52242}
\email{hao-fang@uiowa.edu}
\author{Wei Wei}
\address{School of Mathematical Sciences, University of Science and Technology
of China, Hefei, China, 230026}
\email{weisx001@mail.ustc.edu.cn}
\thanks{H.F.'s works is partially supported by a Simons Foundation research
collaboration grant. W.W.'s works is supported by China Scholarship
Council .}
\maketitle

\section{Introduction}

In this paper we study the $\sigma_{2}$ Yamabe problem on conic four-spheres.
First we fix our notations. For a closed Riemannian manifold $(M^{n},g)$,
let Ric, $R$ and $A$ be the corresponding Ricci curvature, Scalar
curvature, and Schouten tensor, respectively. That means,
\begin{align*}
A & =\frac{1}{n-2}(\text{{\rm Ric}}-\frac{R}{2(n-1)}g),\\
R & =g^{ij}{\rm Ric}_{ij}.
\end{align*}
Let $\{\lambda(A)\}_{i=1}^{n}$ be the set of eigenvalues of $A$
with respect to $g$. Define $\sigma_{k}(g^{-1}A_{g})$ to be the
$k-$th symmetric function of $\{\lambda(A)\},$ 

\[
\sigma_{k}(g^{-1}A_{g})=\sum_{1\le i_{1}\cdots<i_{k}\le n}\lambda_{i_{1}}\lambda_{i_{2}}\cdots\lambda_{i_{k}}.
\]
We note that $\sigma_{1}(\lambda(A))={\rm Tr}A=\frac{1}{2(n-1)}R$.
When necessary, we use, for example, $R_{g}$ to denote $R,$ when
the metric needs to be specified. We define the smooth conformal metric
class $[g]=\{g_{u}=e^{2u}g,\ u\in C^{\infty}(M)\}.$ 

The classical Yamabe problem is to look for constant scalar curvature
metrics in a given conformal metric class, which has been completed
solved through works of Yamabe \cite{Y}, Trudinger \cite{Tru}, Aubin
\cite{A}, and Schoen \cite{S}. 

In \cite{V1}, Viaclovsky raised the $\sigma_{k}$ Yamabe problem,
which is to look for constant $\sigma_{k}(g^{-1}A_{g})$ curvature
metric in a given conformal class. In particular, for $k\geq2,$ the
corresponding fully nonlinear equation is the following 
\begin{equation}
\sigma_{k}(\lambda(A_{e^{2u}g}))=constant.\label{eq:k-general equation}
\end{equation}

To fully understand the relation between solutions of (\ref{eq:k-general equation})
and conformal geometry of $M,$ a so-called positive cone condition
is commonly required. We define 
\[
\mathcal{C}_{k}^{+}:=\{g,\ \sigma_{1}(A_{g})>0,\cdots,\,\sigma_{k}(A_{g})>0\}.
\]
Note that when $g\in\mathcal{C}_{k}^{+}$, equation (\ref{eq:k-general equation})
becomes a fully nonlinear elliptic partial differential equation when
$k>1.$ 

Of all cases, $\sigma_{2}$ Yamabe problem for four-manifold is of
particular interest. From an analytical point of view, this problem
is variational \cite{V1,BJ,STW}. While from a geometric point of
view, $\sigma_{2}(A)$ in four-manifold is connected with the Gauss-Bonnet-Chern
integrand. For a closed four-manifold, we have the following 
\[
k_{g}:=\int_{M}\sigma_{2}(A)dv_{g}=2\pi^{2}\chi(M)-\frac{1}{16}\int_{M}|W|^{2}dv_{g},
\]
where $W$ is the Weyl tensor of $g$ and $\chi(M)$ is the Euler
characteristic of $M$.

Chang-Gursky-Yang in \cite{CGY1} have proved that: when four-manifold
$M$ has a positive Yamabe constant and positive $k_{g}$, there exists
a conformal metric $g_{u}=e^{2u}g$ such that $\sigma_{2}(A_{g})>0.$
Furthermore, in \cite{CGY2}, they have proved that under the same
condition, for any prescribed function $f>0$, there exists a smooth
solution of (\ref{eq:k-general equation}) when $M$ is not conformally
equivalent to the round sphere. When the manifold is conformally flat,
Viacolvsky \cite{V1} has shown that under natural growth conditions
the solution is unique up to a translation. Gursky-Streets \cite{GS}
prove that solutions are unique unless the manifold is conformally
equivalent to the standard round sphere. For the conformal class of
standard round 4-sphere, see \cite{CGY3,V1,V4,LL1,LL2,CHY,CHY2}.

Many related works further explore the relation between the geometry/topology
of four-manifolds and solutions of $\sigma_{2}$ Yamabe problem or
related PDEs. See, for example, \cite{CGY4,CQY,GV5}. Also, many works
are devoted to the general $\sigma_{k}$ Yamabe problem for $k\ge2$.
We refer readers to \cite{GV0,GV1,GV2,GV3,GW1,GW2,GB,SC,L,L1,LL1,LL2,LN,PVW,STW,TW,V1,V2,V3,V4,W1,SS,BG}
and references therein.

At the same time, singular sets of locally conformally flat metrics
with positive $\sigma_{k}$ curvature are widely studied. In \cite{CHFY},
Chang, Hang, and Yang proved that if $\Omega\subset S^{n}(n\ge5)$
admits a complete, conformal metric $g$ with $\sigma_{1}(A_{g})\ge c>0,\sigma_{2}(A_{g})\ge0$,
and $|R_{g}|+|\nabla_{g}R|_{g}\le c_{0}$, then $\dim(S^{n}\backslash\Omega)<(n-4)/2.$
This has been further generalized by Gon$\acute{z}$alez \cite{G1}
and Guan-Lin-Wang \cite{GLW} to the case of $2<k<n/2$. More related
works can be seen in \cite{G0,G2,L} and their references.

In this article, we concentrate on four-manifolds with isolated singularities.
In particular, we concentrate on conic manifolds. For a closed manifold
$M$ with a smooth background Riemannian metric $g_{0},$ a singular
metric $g_{1}$ is said to have a conic singularity of order $\beta$
at a given point $p\in M$, if in a local coordinate centered at $p$,
$g_{1}=e^{2v}d(x,p)^{2\beta}g_{0}$, where $v$ is a locally bounded
function and $d(x,p)$ is the distance function with respect to $g_{0}.$ 

To better describe our singularity, we define a conformal divisor
\[
D=\sum_{i}^{q}\beta_{i}p_{i},
\]
where $p_{1},\cdots p_{q}\in M$ are $q$ distinct points, $q\in\mathbb{N},$
and $\beta_{1},\cdots,\beta_{q}\in(-1,0).$ A conic metric $g_{1}$
is said to represent the divisor $D$ with respect to the background
metric $g_{0}$, if $g_{1}$ has conic singularity of order $\beta_{i}$
at $p_{i}$ with respect to $g_{0}$ and is smooth elsewhere. We also
define the singular conformal metric class $[g_{D}]$ to be the collection
of all such metrics. Note that $[g_{D}]$ is dependent on the triple
$(M,g_{0},D)$. We call $g_{1}\in C_{k}^{+}$ if for all $x\in M\backslash\{p_{1},\cdots,p_{q}\},$
$\sigma_{1}(A_{g_{1}})>0,\cdots\sigma_{k}(A_{g_{1}})>0.$

As an example, let us describe conic metrics on standard spheres.
We may use the Euclidean model by the standard stereographic projection.
Let $g_{E}$ be the standard Euclidean metric on $\mathbb{\mathbb{R}}^{n}.$
A conic metric $g_{u}=e^{2u}g_{E}$ representing a conformal divisor
$D=\sum_{i}^{q}\beta_{i}p_{i}$, where $p_{i}\in\mathbb{R}^{n},$
$i=1,\cdots,q-1$ , $p_{q}=\infty,$ and all $\beta_{i}\in(-1,0),$
if and only if $u$ satisfies the following
\begin{itemize}
\item $u(x)=\beta_{i}\ln|x-p_{i}|+v_{i}(x)$ as $x\rightarrow p_{i}$ for
$i=1,\cdots,\,q-1;$ 
\item $u(x)=(-2-\beta_{q})\ln|x|+v_{\infty}(x)$ as $|x|\rightarrow\infty$, 
\end{itemize}
where $v_{i}(x)$ and $v_{\infty}(x)$ are bounded in their respective
neighborhoods.

Chang-Han-Yang \cite{CHY} have classified global radial solutions
to $\sigma_{k}$ Yamabe problem. Especially, in the case $k=\frac{n}{2},$
solutions of (\ref{eq:k-general equation}) have two conic points
with identical cone angles, which we call American footballs. Li has
studied the radial symmetry of solutions in the punctured Euclidean
space in \cite{L}, which, by our definition, is equivalent to the
study on the conic sphere with two isolated singular points. These
results imply that a conical sphere with two singularities and a constant
$\sigma_{n/2}$ curvature can only be a football.

Locally, a deep theorem of Han-Li-Teixeira \cite{HLT} describes the
behavior of the conformal factor $u$ near the singularity when the
$\sigma_{k}$ curvature is constant.
\begin{thm}
\label{C=00003D00005Calpha regularity} \cite{HLT} Let u(x) be a
smooth solution of \textup{$\sigma_{\frac{n}{2}}(g^{-1}A_{g})=c$}
on $B_{R}\backslash\{0\}$, where $g=e^{2u}g_{E}\in C_{n/2}^{+}$,
$c$ is a positive constant and $n$ is even. Then there exists some
constant $\beta$ with $-1<\beta\le0$ and a $C^{\alpha}$ function
$v(x)$ such that $u(x)=v(x)+\beta\log|x|$ and $v(0)=0.$ 
\end{thm}

When $k=1$, the above theorem was first proved in \cite{CGS}, where
Caffarelli-Gidas-Spruck use the radial average to approximate the
solution. Many related works followed. In particular, Korevaar-Mazzeo-Pacard-Schoen
\cite{KMPS} develop analysis of linearized operators at these global
singular solutions and give an alternative proof. Han-Li-Teixeira
\cite{HLT} have successfully applied the methods of Caffarelli-Gidas-Spruck
\cite{CGS} and Korevaar-Mazzeo-Pacard-Schoen \cite{KMPS}, to the
$\sigma_{k}$ Yamabe problem $(k\ge2$) and described the regularity
of the singular solution near isolated singularity. Theorem \ref{C=00003D00005Calpha regularity}
is a special case of their main results. One of our motivations for
this work is to obtain a global description of solutions of $\sigma_{\frac{n}{2}}(g^{-1}A_{g})=c$
. As we can see later, Theorem \ref{C=00003D00005Calpha regularity}
is also the starting point of our analysis.

Another motivation for our study is from the conclusion of conic surfaces
by Troyanov. In his now classical work \cite{Tr}, Troyanov presented
the following
\begin{defn}
\label{def:old}Let $S$ be a conic 2-sphere with divisor $D=\sum_{i=1}^{q}\beta_{i}p_{i}.$
Define $\chi(S,D):=2+\sum_{i=1}^{q}\beta_{i}$, where $2$ is the
Euler characteristic of the 2-sphere. Then 
\end{defn}

\begin{itemize}
\item $(S,D)$ is called subcritical if $0<\chi(S,D)<\min\{2,2+2\min_{1\le i\le q}\beta_{i}\};$
\item $(S,D)$ is called critical if $0<\chi(S,D)=\min\{2,2+2\min_{1\le i\le q}\beta_{i}\};$
\item $(S,D)$ is called supercritical if $\chi(S,D)>\min\{2,2+2\min_{1\le i\le q}\beta_{i}\}>0.$
\end{itemize}
Accordingly, Troyanov proved the existence of a unique solution to
the conic Yamabe problem when $(S,D)$ is subcritical. Later Luo-Tian
\cite{LT} have showed that the subcritical condition is both necessary
and sufficient when $q\ge3.$ When $q=2,$ Chen-Li \cite{CL1} have
proved that only the case $\beta_{1}=\beta_{2}$ has the solution
and the corresponding manifold is the football.

Note that Troyanov's theory introduces the supercritical case, which
does not exist in the smooth category. In some sense, a sphere is
the only critical case when $M$ is smooth. Troyanov also has studied
surfaces of higher genus. See \cite{Tr} for more details.

Our first result is a Gauss-Bonnet-Chern formula for conic spheres
of general even dimension $n$ and constant positive $\sigma_{m}$
curvature, where $m=\frac{n}{2}.$
\begin{thm}
\label{thm:Gauss-bonnet-chern-1-1}Suppose that $g=e^{2u}g_{0}$ is
a conic metric on the round sphere in the class $(S^{n},g_{0},D)$
such that $g_{u}\in C_{m}^{+}$ with constant $\sigma_{m}$.

Then we have
\begin{equation}
\frac{1}{|S_{n-1}|}\int_{S^{n}}\frac{m}{2^{m-1}}\sigma_{m}(g^{-1}A_{g})dv_{g}=2-\sum_{i=1}^{q}f(\beta_{i}),\label{eq:gauss-bonnet-chern equality-1-1}
\end{equation}

where 
\[
f(\beta)=\frac{1}{2^{n-2}}\sum_{k=0}^{m-1}(\begin{array}{c}
n-1\\
k
\end{array})(2+\beta)^{k}|\beta|^{n-k-1}.
\]
\end{thm}

While $\sigma_{m}$ is known as the Phaffian curvature for smooth
locally conformally flat manifolds, (\ref{eq:gauss-bonnet-chern equality-1-1})
gives precise Gauss-Bonnet-Chern defect for each conic singular point.
Note that when $n=2,$ equation (\ref{eq:gauss-bonnet-chern equality-1-1})
is the classical Gauss-Bonnet-Chern formula obtained by Troyanov \cite{Tr}. 

For $m=2$, $n=4$, (\ref{eq:gauss-bonnet-chern equality-1-1}) becomes
\[
\frac{1}{|S_{3}|}\int_{S^{4}}\sigma_{2}(g^{-1}A_{g})dv_{g}=2-\sum_{i=1}^{q}\frac{\beta_{i}^{3}+3\beta_{i}^{2}}{2}.
\]

Theorem 3 is true under weaker conditions. We have proved a slightly
stronger statement. See Section 2, Theorem \ref{thm:Gauss-bonnet-chern-1}. 

While it is tempting to use the Guass-Bonnet-Chern integral to classify
conic 4-spheres as Troyanov has done in 2-dimension, the 4-dimension
case is far more complicated. In order to obtain a proper definition
for conic manifolds of dimension 4 or higher, we consider the case
of conic 4-spheres with the standard round sphere background metric.
While it is geometrically and topologically simple, it is often the
most difficult case in terms of analysis, which is indicated by previous
studies of the smooth case. To state our main result, we give the
following definition:
\begin{defn}
\label{def:New}Let $(S^{4},D,g_{S^{4}})$ be a conic 4-sphere with
the standard round background metric $g_{S^{4}}=\frac{4}{(1+|x|^{2})^{2}}g_{E}$.
For all $j=1,\cdots,q$, we denote $\widetilde{\beta_{j}}:=\bigg(\sum_{1\le i\neq j\le q}\beta_{i}^{3}\bigg)^{1/3}$. 
\end{defn}

\begin{itemize}
\item We call $(S^{4},D)$ subcritical for $\sigma_{2}$ Yamabe equation
if for any $j=1,\cdots,q$\\
 $\frac{3}{8}\beta_{j}^{2}(\beta_{j}+2)^{2}<\frac{3}{8}\widetilde{\beta_{j}}^{2}(\widetilde{\beta_{j}}+2)^{2}+(\widetilde{\beta_{j}}+\frac{3}{2})(\sum_{1\le i\neq j\le q}\beta_{i}^{2}-\widetilde{\beta_{j}}^{2}),$ 
\item We call $(S^{4},D)$ critical for $\sigma_{2}$ Yamabe equation if
there exists a $j\in\{1,\cdots,q\}$ such that\\
 $\frac{3}{8}\beta_{j}^{2}(\beta_{j}+2)^{2}=\frac{3}{8}\widetilde{\beta_{j}}^{2}(\widetilde{\beta_{j}}+2)^{2}+(\widetilde{\beta_{j}}+\frac{3}{2})(\sum_{1\le i\neq j\le q}\beta_{i}^{2}-\widetilde{\beta_{j}}^{2})$, 
\item Otherwise, we call $(S^{4},D)$ supercritical for $\sigma_{2}$ Yamabe
equation, which means that there exists a $j\in\{1,\cdots,\,q\},$\\
 $\frac{3}{8}\beta_{j}^{2}(\beta_{j}+2)^{2}>\frac{3}{8}\widetilde{\beta_{j}}^{2}(\widetilde{\beta_{j}}+2)^{2}+(\widetilde{\beta_{j}}+\frac{3}{2})(\sum_{1\le i\neq j\le q}\beta_{i}^{2}-\widetilde{\beta_{j}}^{2}).$
\end{itemize}
We remark that Definition \ref{def:New} is purely numerical and independent
of the geometric configuration of points $\{p_{i}\}$. 

Finally, we present our main theorem.
\begin{thm}
\label{thm:Main theorem-1} Let $(S^{4},\,D,\ g_{0})$ be defined
as above. Assume that $g\in\mathcal{C}_{2}^{+}$. If $(S^{4},D)$
is supercritical, there does not exist a conformal metric $g\in[g_{D}]$
with constant $\sigma_{2}$ curvature. If $(S^{4},D)$ is critical
with constant $\sigma_{2}$ curvature, then $(S^{4},g)$ is a football
as defined in \cite{CHY}. 
\end{thm}

Theorem \ref{thm:Main theorem-1} justifies Definition \ref{def:New},
while both should be considered as a 4-dimensional generalization
of Troyanov's Definition \ref{def:old} and corresponding results
in \cite{Tr,LT,CL1}. As noted earlier, when the metric is smooth,
i.e., $D$ is empty, Theorem \ref{thm:Main theorem-1} was first proved
by Viaclovsky \cite{V1}; when $D$ contains 1 or 2 points, Theorem
\ref{thm:Main theorem-1} was first proved by Li \cite{L}. Our approach
is different from those of earlier works and may be considered as
an alternative proof. For general $D$, Theorem \ref{thm:Main theorem-1}
is new.

Our method to prove this theorem is to do a careful analysis of geometric
quantities on level sets of the conformal factor $u$. Such an approach
was first used in the 2-dimensional case by \cite{B,ChLin,CL,CL1,BDM,BM,BD}
and later used by the first named author and Lai \cite{FL1,FL2,FL3}
to obtain various sharp geometric bounds. However, $\sigma_{2}$ Yamabe
equation is now a fully nonlinear equation of Monge-Amp$\grave{e}$re
type. New quantities and more importantly, new ideas to treat the
non-linearity are needed. For dimension 4, we find a surprising integrable
quantity which leads to a new monotonicity formula that is key to
our proof. 

We would like to remark that Troyanov's theory on conic surfaces may
be interpreted both in conformal geometry and K$\ddot{a}$hler geometry.
In particular, Definition \ref{def:old} is identical to stability
conditions on singular algebraic curves. See \cite{CHY} for details.
The recently settled Yau-Tian-Donaldson conjecture connects the existence
of the K$\ddot{a}$hler-Einstein metric to the stability conditions
\cite{CDS1,CDS2,CDS3,T1,T2}. The study of conic metrics on K$\ddot{a}$hler
manifolds with singularities over algebraic divisors is crucial in
the solution of the Yau-Tian-Donaldson conjecture. It is our intention
to develop a parallel theory in terms of conformal geometry. See also
Fang-Ma \cite{FM}, where Branson's Q curvature is considered in dimension
4.

In a subsequent work, we would like to address the $\sigma_{2}$ Yamabe
problem for subcritical conic 4-spheres. As we have only treated locally
conformally flat manifolds in this paper, we hope that in future works,
we can generalize the corresponding definitions and results to the
general 4-manifolds. Higher dimensional case is another direction
that we would like to explore.

The first named author would like to thank Professor Paul Yang for
discussions that motivated this work. The second named author would
like to thank the University of Iowa for hospitality during her stay
from September 2017 to May 2019. Both authors would like to thank
the anonymous referee for careful reading of an earlier version of
the paper. 

We organize the paper as follows. In Section 2, we prove the Gauss-Bonnet-Chern
formula for conic conformally flat manifold. In Section 3, we introduce
the quantities related to the level set and obtain some fundamental
equalities. In Section 4, we establish a key inequality and then prove
our main theorem.

\section{Gauss-Bonnet-Chern formula}

In this Section, we prove a Gauss-Bonnet-Chern formula for conic spheres.
Here we consider any even natural number $n$.

On a Riemann manifold $(M,g_{0})$ of dimension $n$. Let $g=e^{2u}g_{0}$
and $Ric_{g},$ $R_{g}$ denote the Ricci and scalar curvature of
$g$. Also let $Ric_{g_{0}}$, $R_{g_{0}}$ denotes the Ricci and
scalar curvature of $g_{0}$. Then under a given local coordinates, 

\begin{align*}
Ric_{g} & =Ric_{g_{0}}-(n-2)\nabla^{2}u-\triangle u\cdot g_{0}+(n-2)du\otimes du-(n-2)|\nabla u|^{2}g_{0}
\end{align*}

\[
R_{g}=e^{-2u}\{R_{g_{0}}-2(n-1)\triangle u-(n-1)(n-2)|\nabla u|^{2}\},
\]

\begin{align}
A_{g} & =\frac{1}{n-2}(Ric_{g}-\frac{R_{g}}{2(n-1)}g)\nonumber \\
 & =\frac{1}{n-2}(Ric_{g_{0}}-\frac{g_{0}}{2(n-1)}R_{g_{0}}-(n-2)\nabla^{2}u+(n-2)du\otimes du-\frac{(n-2)}{2}|\nabla u|^{2}g_{0})\nonumber \\
 & =A_{g_{0}}-\nabla^{2}u+du\otimes du-\frac{1}{2}|\nabla u|^{2}g_{0},\label{eq:general form}
\end{align}
where all derivatives of $u$ are with respect to $g_{0}.$

Let $g_{E}$ be the standard Euclidean metric on $\mathbb{R}^{n}.$
It is known that the standard metric on $S^{n}$ can be represented
as $(S^{n},\,g_{S^{n}})$, where 
\[
g_{S^{n}}=\frac{4}{(1+|x|^{2})^{2}}g_{E}
\]
 and $x=(x_{1},\cdots,x_{n})$ is the coordinate function of $\mathbb{R}^{n}$
with $|x|$ being its Euclidean norm. Let $|S_{n}|$ be the volume
of $S^{n}$ with respect to $g_{S^{n}}.$ For future use, we note
that for $n=2m$, 
\[
|S_{n}|=|S_{n-1}|\frac{2^{n-1}(m-1)!^{2}}{(n-1)!}.
\]

Let $(S^{n},\,g_{u},\,D)$ be a conical sphere as defined in the Introduction.
We assume that $g_{u}\in\mathcal{C}_{m}^{+},$ where $m=\frac{n}{2}.$
Note that for $g=g_{u}=e^{2u}g_{E},$ we have the following:

\[
{\rm Ric}=-(n-2)\nabla^{2}u+(n-2)du\otimes du+(-\triangle u-(n-2)|\nabla u|^{2})g_{E},
\]

\[
R=e^{-2u}(-2(n-1)\triangle u-(n-1)(n-2)|\nabla u|^{2}),
\]
and 
\[
A_{g}=-\nabla^{2}u+du\otimes du-\frac{|\nabla u|^{2}}{2}g_{E}.
\]
Here all derivatives are with respect to the Euclidean metric $g_{E}$. 

Similarly, we have the following:
\[
\sigma_{k}(g^{-1}A_{g})=\frac{1}{k!}\sum_{i_{1},\cdots,i_{k},j_{1,}\cdots,j_{k}=1}^{n}\delta\big(\begin{array}{ccc}
i_{1} & \cdots & i_{k}\\
j_{1} & \cdots & j_{k}
\end{array}\big)A_{i_{1}}^{j_{1}}\cdots A_{i_{k}}^{j_{k}}.
\]
We define

\[
T_{l}(g^{-1}A)_{j}^{i}=\frac{1}{l!}\sum_{i_{1}\cdots i_{l},i,j_{1},\cdots j_{l},j=1}^{n}\delta(\begin{array}{cccc}
i_{1} & \cdots & i_{l} & i\\
j_{1} & \cdots & j_{l} & j
\end{array})A_{i_{1}}^{j_{1}}\cdots A_{i_{l}}^{j_{l}},
\]
where $A_{i_{1}}^{j_{1}}=g^{j_{1}k_{1}}A_{k_{1}i_{1}}$ and $\delta$
is the Kronecker delta function.

In the rest of this paper, we consider $(S^{n},g_{u},D)$ satisfying
the following curvature equation
\begin{equation}
\sigma_{m}(g^{-1}A_{g})=(\begin{array}{c}
n\\
m
\end{array})(\frac{1}{2})^{m}.\label{k-equation}
\end{equation}
Notice that the right hand side is the corresponding value of the
standard sphere $(S^{n},\,\frac{4}{(1+|x|^{2})^{2}}|dx|^{2})$.

It is well known that in the conformally flat case, the Gauss-Bonnet-Chern
integral is just $\sigma_{m}$ up to a constant. In addition, it has
a special divergence structure as follows (see also \cite{H}):
\begin{lem}
\label{lem:divform} On a conformally flat manifold $(M^{n},\,g)$,
\[
\sigma_{m}(g^{-1}A_{g})=-\frac{1}{m}{\rm div}_{g}\{\sum_{j=1}^{m}\frac{T_{m-j}(g^{-1}A_{g})_{b}^{a}|\nabla_{g}u|_{g}^{2(j-1)}}{2^{j-1}}\nabla_{g}^{b}u\}.
\]
Especially, for $n=4$, $m=2$, we have 
\[
\sigma_{2}(g_{E}^{-1}A_{g})=-\frac{1}{2}\partial_{i}((-\triangle u\delta_{ij}+u_{ij}-u_{i}u_{j})u_{j}),
\]
where all derivatives are with respect to the Euclidean metric of
$\mathbb{R}^{4}.$
\end{lem}

Before we establish the Gauss-Bonnet-Chern formula, we first need
to know the asymptotic behavior of conformal factor $u$ near singularities. 
\begin{lem}
\label{lem:asy of u}Assume $(S^{n},g_{u},D)$ has positive constant
$\sigma_{m}$ curvature and $g_{u}\in C_{m}^{+}$; then we have, for
$1\le l\le p-1$, as $|x-p_{l}|\rightarrow0$,

\begin{equation}
u_{i}(x)=\frac{\beta_{l}}{|x-p_{l}|^{2}}(x_{i}-p_{l,i})+o(\frac{1}{|x-p_{l}|}),\label{eq:firstderivative-1}
\end{equation}

\begin{equation}
u_{ij}(x)=\beta_{l}\frac{\delta_{ij}}{|x-p_{l}|^{2}}-2\beta_{l}\frac{(x_{i}-p_{l,i})(x_{j}-p_{l,j})}{|x-p_{l}|^{4}}+o(\frac{1}{|x-p_{l}|^{2}}),\label{eq:secondderivative-1}
\end{equation}

\begin{align}
H(x) & =\frac{3}{|x-p_{l}|}+o(\frac{1}{|x-p_{l}|}),\label{eq:mean curvature-1}
\end{align}
where $H(x)$ is the mean curvature of level set $\{x,\,u(x)=t\}$
near $p_{l}$ and $t$ is sufficiently large. 
\end{lem}

\begin{proof}
These are direct consequences of the main theorem of Han-Li-Teixeira
\cite{HLT}. By Theorem \ref{C=00003D00005Calpha regularity} (Theorem
1 or Corollary 1 in \cite{HLT}), we have for some small $R_{0}>0,$
\[
|x-p_{l}|^{k}(u(x)-\beta_{l}\log|x-p_{l}|)\in C^{k,\,\alpha}(B_{R_{0}}(p_{l})).
\]
 Near $p_{l}$, we then have
\begin{align*}
 & \partial_{i}\big(|x-p_{l}|(u(x)-\beta_{l}\log|x-p_{l}|)\big)\\
 & =\frac{x_{i}-p_{l,i}}{|x-p_{l}|}(u(x)-\beta_{l}\log|x-p_{l}|)+|x-p_{l}|\partial_{x_{i}}(u(x)-\beta_{l}\log|x-p_{l}|)\\
 & =o(1),\quad|x-p_{l}|\rightarrow0.
\end{align*}
Thus, we get 
\begin{equation}
u_{i}(x)=\frac{\beta_{l}}{|x-p_{l}|^{2}}(x_{i}-p_{l,i})+o(\frac{1}{|x-p_{l}|}),\quad|x-p_{l}|\rightarrow0.\label{eq:firstderivative}
\end{equation}
A similar computation shows that as $|x-p_{l}|\rightarrow0,$
\begin{equation}
u_{ij}(x)=\beta_{l}\frac{\delta_{ij}}{|x-p_{l}|^{2}}-2\beta_{l}\frac{(x_{i}-p_{l,i})(x_{j}-p_{l,j})}{|x-p_{l}|^{4}}+o(\frac{1}{|x-p_{l}|^{2}}).\label{eq:secondderivative}
\end{equation}
Furthermore, the mean curvature of level set $\{x:\,u(x)=t\}$ for
a sufficiently large $t$ can be computed as below
\begin{align}
H(x) & =-{\rm div}(\frac{\nabla u}{|\nabla u|})\nonumber \\
 & =\frac{3}{|x-p_{l}|}+o(\frac{1}{|x-p_{l}|})\quad as\,|x-p_{l}|\rightarrow0,\label{eq:mean curvature}
\end{align}
for the singular point $p_{l},$ $l=1,\cdots,q-1.$ We have finished
the proof.
\end{proof}
The main result of this section is the following Gauss-Bonnet-Chern
formula with defects.
\begin{thm}
\label{thm:Gauss-bonnet-chern-1}

Suppose that $g=e^{2u}g_{E}$ is a conic metric on the round sphere
in the class $(S^{n},g_{S^{4}},D)$ such that (\ref{eq:firstderivative-1}),
(\ref{eq:secondderivative-1}) and (\ref{eq:mean curvature-1}) hold
near each point of $D$. Then we have
\begin{equation}
\frac{1}{|S_{n-1}|}\int_{S^{n}}\frac{m}{2^{m-1}}\sigma_{m}(g^{-1}A_{g})dv_{g}=2-\sum_{l=1}^{q}f(\beta_{l}),\label{eq:gauss-bonnet-chern equality-1}
\end{equation}
where 
\[
f(\beta)=\frac{1}{2^{n-2}}\sum_{k=0}^{m-1}(\begin{array}{c}
n-1\\
k
\end{array})(2+\beta)^{k}|\beta|^{n-k-1}.
\]

In particular, when $g$ is a conic metric $(S^{n},g_{S^{n}},D)$
such that $g\in C_{m}^{+}$ with positive constant $\sigma_{m}$ curvature,
then (\ref{eq:gauss-bonnet-chern equality-1}) holds.
\end{thm}

\begin{proof}
By Lemma \ref{lem:divform},

\begin{align*}
 & \int_{S^{n}\backslash\cup_{i}\{p_{i}\}}m\sigma_{m}(g^{-1}A_{g})dvol_{g}\\
= & \lim_{\varepsilon\rightarrow0}\sum_{l=1}^{q-1}\int_{\{|x-p_{l}|=\varepsilon\}}\sum_{k=1}^{m}\frac{T_{m-k}(A)_{b}^{a}|\nabla u|^{2(k-1)}}{2^{k-1}}u_{b}\nu_{a}ds\\
 & -\lim_{R\rightarrow\infty}\int_{\{x,\,|x|=R\}}\sum_{k=1}^{m}\frac{T_{m-k}(A)_{b}^{a}|\nabla u|^{2(k-1)}}{2^{k-1}}u_{b}\nu_{a}ds\\
:= & \lim_{\varepsilon\rightarrow0}\sum_{l=1}^{q-1}I_{\varepsilon}^{l}-\lim_{R\rightarrow\infty}I_{R},
\end{align*}
where 
\[
I_{\varepsilon}^{l}=\int_{\{|x-p_{l}|=\varepsilon\}}\sum_{k=1}^{m}\frac{T_{m-k}(A)_{b}^{a}|\nabla u|^{2(k-1)}}{2^{k-1}}u_{b}\nu_{a}ds,
\]

and 

\[
I_{R}=\int_{|x|=R}\sum_{k=1}^{m}\frac{T_{m-k}(A)_{b}^{a}|\nabla u|^{2(k-1)}}{2^{k-1}}u_{b}\nu_{a}ds.
\]

For each $l,$ on the hypersurface $\{x,|x-p_{l}|=\varepsilon\}$,
$\nu=(\frac{x_{1}-p_{l,1}}{|x-p_{l}|},\cdots,\:\frac{x_{n}-p_{l,n}}{|x-p_{l}|})$.
By Theorem \ref{C=00003D00005Calpha regularity}, there exists a $C^{\alpha}$
function $v_{l}$ satisfying $v_{l}(p_{l})=0$ and $u(x)-\beta_{l}\log|x-p_{l}|=v_{l}(x).$
By Lemma \ref{lem:asy of u}, we have 

\begin{align*}
 & -u_{ij}+u_{i}u_{j}-\frac{|\nabla u|^{2}}{2}\delta_{ij}\\
 & =(-\beta_{l}-\frac{\beta_{l}^{2}}{2})\frac{\delta_{ij}}{|x-p_{l}|^{2}}+(2\beta_{l}+\beta_{l}^{2})\frac{(x_{i}-p_{l,i})(x_{j}-p_{l,j})}{|x-p_{l}|^{4}}+\frac{o(1)}{|x-p_{l}|^{2}}.
\end{align*}
Thus we get

\begin{align*}
 & \sum_{a,b=1}^{n}\frac{T_{m-k}(A)_{b}^{a}|\nabla u|^{2(k-1)}}{2^{k-1}}u_{b}\nu_{a}\\
= & \frac{1}{2^{k-1}(m-k)!}\sum_{\begin{array}{c}
i_{1}\cdots i_{m-k,}a\\
j_{1,\cdots}j_{m-k},b=1
\end{array}}^{n}\delta(\begin{array}{cccc}
i_{1} & \cdots & i_{m-k} & a\\
j_{1} & \cdots & j_{m-k} & b
\end{array})\\
 & \cdot\bigg((-\beta_{l}-\frac{\beta_{l}^{2}}{2})\frac{\delta_{i_{1}j_{1}}}{|x-p_{l}|^{2}}+(2\beta_{l}+\beta_{l}^{2})\frac{(x_{i_{1}}-p_{l,i_{1}})(x_{j_{1}}-p_{l,j_{1}})}{|x-p_{l}|^{4}}+\frac{o(1)}{|x|^{2}}\bigg)\cdots\\
 & \cdot\bigg((-\beta_{l}-\frac{\beta_{l}^{2}}{2})\frac{\delta_{i_{m-k}j_{m-k}}}{|x-p_{l}|^{2}}+(2\beta_{l}+\beta_{l}^{2})\frac{(x_{i_{m-k}}-p_{l,i_{m-k}})(x_{j_{m-k}}-p_{l,j_{m-k}})}{|x-p_{l}|^{4}}+\frac{o(1)}{|x-p_{l}|^{2}}\bigg)\\
 & \cdot\big(\beta_{l}\frac{(x_{a}-p_{l,a})(x_{b}-p_{l,b})}{|x-p_{l}|^{3}}+\frac{o(1)}{|x-p_{l}|}\big)(\beta_{l}^{2}\frac{1}{|x-p_{l}|^{2}}+\frac{o(1)}{|x-p_{l}|^{2}})^{k-1}\\
= & \frac{1}{2^{k-1}(m-k)!}\sum_{\begin{array}{cc}
i_{1}\cdots i_{m-k},a,b\\
j_{1,\cdots}j_{m-k}=1
\end{array}}^{n}\delta(\begin{array}{cccc}
i_{1} & \cdots & i_{m-k} & a\\
j_{1} & \cdots & j_{m-k} & b
\end{array})(-\beta_{l}-\frac{\beta_{l}^{2}}{2})\frac{\delta_{i_{1}j_{1}}}{|x-p_{l}|^{2}}\cdots\cdot\\
 & (-\beta_{l}-\frac{\beta_{l}^{2}}{2})\frac{\delta_{i_{m-k}j_{m-k}}}{|x-p_{l}|^{2}}\beta_{l}\frac{(x_{a}-p_{l,a})(x_{b}-p_{l,b})}{|x-p_{l}|^{3}}(\beta_{l}^{2}\frac{1}{|x-p_{l}|^{2}})^{k-1}+o(\frac{1}{|x-p_{l}|^{n-1}})\\
= & (-\beta_{l}-\frac{\beta_{l}^{2}}{2})^{m-k}\beta_{l}^{2(k-1)+1}\frac{(n-1)(n-2)\cdots(n-m+k)}{2^{k-1}(m-k)!|x-p_{l}|^{n-1}}+o(\frac{1}{|x-p_{l}|^{n-1}}).
\end{align*}
 Therefore,
\begin{align*}
\\
\lim_{\epsilon\to0}\sum_{l=1}^{q-1}I_{\epsilon}^{l}= & \lim_{\varepsilon\rightarrow0}\sum_{l=1}^{q-1}\int_{\{|x-p_{l}|=\varepsilon\}}\sum_{k=1}^{m}\frac{T_{m-k}(A)_{b}^{a}|\nabla u|^{2(k-1)}}{2^{k-1}}u_{b}\nu_{a}ds\\
= & \lim_{\varepsilon\rightarrow0}\sum_{l=1}^{q-1}\sum_{k=1}^{m}[(-\beta_{l}-\frac{\beta_{l}^{2}}{2})^{m-k}\beta_{l}^{2(k-1)+1}\frac{(n-1)(n-2)\cdots(n-m+k)}{2^{k-1}(m-k)!|\varepsilon|^{n-1}}+o(\frac{1}{|\varepsilon|^{n-1}})]\varepsilon^{n-1}|S_{n-1}|\\
= & \sum_{l=1}^{q-1}\sum_{k=1}^{m}[(-\beta_{l}-\frac{\beta_{l}^{2}}{2})^{m-k}\beta_{l}^{2(k-1)+1}\frac{(n-1)(n-2)\cdots(n-m+k)}{2^{k-1}(m-k)!}|S_{n-1}|\\
= & \sum_{l=1}^{q-1}\sum_{k=1}^{m}(\begin{array}{c}
n-1\\
m-k
\end{array})(\frac{1}{2})^{m-1}(-2\beta_{l}-\beta_{l}^{2})^{m-k}\beta_{l}^{2(k-1)+1}|S_{n-1}|\\
= & \sum_{l=1}^{q-1}\sum_{k=1}^{m}(\begin{array}{c}
n-1\\
m-k
\end{array})(\frac{1}{2})^{m-1}(-2-\beta_{l})^{m-k}\beta_{l}^{m+k-1}|S_{n-1}|\\
= & |S_{n-1}|\sum_{l=1}^{q-1}(-2^{m-1}f(\beta_{l})).
\end{align*}
As $u(x)=(-2-\beta_{q})\log|x|+v_{\infty}(x),\,as\,|x|\rightarrow\infty,$
where $v_{\infty}$ is locally bounded, by a similar computation which
we will omit here, we get that 
\[
\lim_{R\rightarrow\infty}I_{R}=|S_{n-1}|(-2^{m-1}f(-2-\beta_{q})).
\]

Finally, to finish the proof of Theorem \ref{thm:Gauss-bonnet-chern-1},
we note the following simple fact 
\[
2=f(-2-\beta)+f(\beta).
\]
 Equality (\ref{eq:gauss-bonnet-chern equality-1}) is thus proved.
The second part of Theorem follows from Lemma 7.
\end{proof}
\begin{rem}
It is obvious that the Gauss-Bonnet-Chern formula for singular manifolds
is highly dependent on the asymptotic behavior of the metric near
the singularity. The $\sigma_{2}$ Yamabe equation is required to
obtain the needed local regularity to compute the Gauss-Bonnet-Chern
defect near each singularity. 
\end{rem}

\section{Level set and Related functions}

In this section, we study global solutions of (\ref{k-equation})
via the level set method. In \cite{HLT}, the authors have proved
that the radial average of the solution is a good approximation to
the solution and satisfies an ODE, which is an approximation to the
ODE satisfied by a radial solution to (\ref{k-equation}). In this
paper, we will instead use some quantities related to the level set
and obtain an ordinary differential inequality, which is inspired
by \cite{FL1,FL2,FL3}. As the $\sigma_{2}$ equation is fully nonlinear
and quantities used in dimension 2 are not sufficient, we introduce
some new functions constructed by curvatures of the level set and
gradient of the solution.

We begin with the following definition about the level sets of the
conformal factor $u:$
\begin{align*}
L(t) & =\{x:u=t\}\subset M,\\
S(t) & =\{x:u\geq t\}\subset M.
\end{align*}

For a fixed $t,$ it is clear that $L(t)$ is smooth at a point $P\in L(t)$
if and only if $\nabla u\neq0.$ Define 
\[
S=\{x\in\mathbb{R}^{4}\backslash\{p_{1},\cdots,p_{q}\}|\,\nabla u(x)=0\}.
\]

\begin{lem}
\label{lem:measure 0}For $U\subset\mathbb{R}^{4},$ if $\sigma_{2}(g_{u})\neq0$
in $U$, then $\mathcal{H}^{3}(S\cap U)=\mathcal{H}^{4}(S\cap U)=0,$
where \textup{$\mathcal{H}^{3}$,} $\mathcal{H}^{4}$ is the 3-dimensional
and 4-dimensional Hausdorff measure. 
\end{lem}

\begin{proof}
For any $P\in S\cap U$, since $\sigma_{2}(g_{u})(P)\neq0$ and $|\nabla u|(P)=0$,
we have $\sigma_{2}(\nabla^{2}u)(P)\neq0.$ Thus, there exist $i,j\in\{1,2,3,4\}$,
$i\neq j$, such that $u_{ii}u_{jj}-u_{ij}^{2}|_{P}\neq0$. Hence,
$\nabla u_{i}(P)$ and $\nabla u_{j}(P)$ are linearly independent.
Thus, by implicit function theorem, $S_{i}=\{u_{i}=0\}$ and $S_{j}=\{u_{j}=0\}$
are both locally smooth and transversal to each other. Hence, for
some small $r_{0},$ we know $S\cap B_{r_{0}}(P)\subset S_{i}\cap S_{j}\cap B_{r_{0}}(P)$
has vanishing $\mathcal{H}^{3}$ and $\mathcal{H}^{4}$ measure. 
\end{proof}
Now we describe our local coordinate choice near a fixed point $P\in\mathbb{R}^{4}\backslash[S\cup\{p_{1},\cdots,p_{q}\}].$
Let $t=u(P)$. We define 
\[
x^{4}(Q)=-{\rm sgn}(u(Q)-t){\rm dist}(Q,L(t))
\]
 for $Q$ near $P.$ Note that this is well defined when $L(t)$ is
smooth at $P$, which is true by our choice of $P$. We also define
local normal coordinate functions $x^{1},x^{2},x^{3}$ on an open
set $V\subset L(t)$ near $P$ and then extend them smoothly to an
open set $U\subset\mathbb{R}^{4}$. Thus, we have got a local coordinate
system $\{x^{i}\},$ $i=1,\cdots,4$ of $\mathbb{R}^{4}$ near $P$
such that $<\frac{\partial}{\partial x^{i}},\frac{\partial}{\partial x^{j}}>|_{P}=\delta_{ij}$,
$<\frac{\partial}{\partial x^{4}},\frac{\partial}{\partial x^{4}}>|_{U}=1$. 

As before, we use $\nabla$ to denote the Levi-Civita connection of
$g_{E}$ and write $u_{i}=\nabla_{\frac{\partial}{\partial x^{i}}}u$
and $\nabla_{ij}u=\nabla_{i}\nabla_{j}u=u_{ij}$. By definition of
$x^{4},$ $\frac{\partial}{\partial x^{4}}|_{V}=-\frac{\nabla u}{|\nabla u|}$.
Especially, we note that $u_{44}$ is independent of choices of $x^{1}$,
$x^{2}$ and $x^{3}$ and well defined on $L(t)\backslash S$. Let
$\nabla_{ab}^{L}u$ be the Hessian of $u$ with respect to the induced
metric on $L(t)$. In the following, $\alpha,\,\beta$ range from
1 to $3$. Notice that $\nabla_{\alpha}^{L}u=0$ on $L(t).$ 

Let $h_{\alpha\beta}$ be the second fundamental form of the level
set $L(t)$ with respect to the outward normal vector. We have the
following Gauss-Weingarten formula

\begin{equation}
\nabla_{\alpha\beta}u=\nabla_{\alpha\beta}^{L}u+h_{\alpha\beta}u_{4}.\label{eq:gauss-weingarten1}
\end{equation}

We may now describe the Schouten tensor using our choice of local
coordinates near $P$. In particular, by (\ref{eq:general form}),
we write $g_{E}^{-1}A_{g}$ as a symmetric matrix as 

\begin{equation}
g_{E}^{-1}A_{g}(P)=\left(\begin{array}{cccc}
 &  &  & -\nabla_{41}u\\
 & h_{\alpha\beta}|\nabla u|-\frac{|\nabla u|^{2}}{2}\delta_{\alpha\beta} &  & -\nabla_{42}u\\
 &  &  & -\nabla_{43}u\\
-\nabla_{41}u & -\nabla_{42}u & -\nabla_{43}u & -\nabla_{44}u+\frac{|\nabla u|^{2}}{2}
\end{array}\right).\label{eq:expression of A}
\end{equation}

For future use, we also define a symmetric $3\times3$ matrix as 
\[
\widetilde{A}(P):=(h_{\alpha\beta}|\nabla u|-\frac{|\nabla u|^{2}}{2}\delta_{\alpha\beta}).
\]

For simplicity, we use $\fint_{L(t)}$ , $\fint_{S(t)}$ to represent
$\frac{1}{|S_{3}|}\varoint_{L(t)}$, $\frac{1}{|S_{3}|}\varoint_{S(t)}$
, respectively. We define

\[
A(t)=\fint_{S(t)}e^{4u}dl
\]
\[
B(t)=\fint_{S(t)}dl,
\]
\[
C(t)=e^{4t}B(t),
\]
and 
\[
z(t)=-(\fint_{L(t)}|\nabla u|^{3}\ dl)^{1/3}.
\]

\begin{align*}
\Sigma_{0}(t) & =\fint_{L(t)}|\nabla u|^{3}\ dl=-z^{3},\\
\Sigma_{1}(t) & =\fint_{L(t)\backslash S}\{2H|\nabla u|^{2}-3|\nabla u|^{3}\}\ dl,
\end{align*}
and finally,

\[
D(t)=\frac{1}{4}(\Sigma_{0}(t)+\Sigma_{1}(t)).
\]
Here $H$ is the mean curvature of the level set $L(t),$ and $dl$
is the induced 3-dimensional measure on $L(t).$ When no confusion
arises, we may omit $dl$. 

Note that while $g_{E}^{-1}A_{g}$ and $\tilde{A}$ are defined by
a local coordinate near a smooth point $P$ of $M$, here $A,B,C,z,\Sigma_{0},\Sigma_{1}$
are independent of the coordinates choice. 

We may now derive behaviors of our new quantities near singularities
by making use of the asymptotic behavior of $u$. A direct consequence
of (\ref{eq:firstderivative-1}), (\ref{eq:secondderivative-1}) and
(\ref{eq:mean curvature-1}) is the following
\begin{lem}
Suppose that $g=e^{2u}g_{E}$ is a conic metric on the round sphere
in the class $(S^{n},D)$ such that (\ref{eq:firstderivative-1})(\ref{eq:secondderivative-1})(\ref{eq:mean curvature-1})
hold near each singular point. We have the following:
\begin{align}
D(+\infty):=\lim_{t\rightarrow+\infty}D(t) & =\frac{1}{4}\sum_{i=1}^{q-1}(|\beta_{i}|^{3}+3(-2\beta_{i}-\beta_{i}^{2})|\beta_{i}|)\nonumber \\
 & =\frac{3}{2}\sum_{i=1}^{q-1}|\beta_{i}|^{2}-\frac{1}{2}\sum_{i=1}^{q-1}|\beta_{i}|^{3},\label{eq:D+infinity}
\end{align}
\end{lem}

\begin{equation}
D(-\infty):=\lim_{t\rightarrow-\infty}D(t)=\frac{3}{2}(2+\beta_{\infty})^{2}-\frac{1}{2}(2+\beta_{\infty})^{3},\label{eq:D-infinity}
\end{equation}

\begin{equation}
\lim_{t\rightarrow+\infty}z(t)=(\sum_{i=1}^{q-1}\beta_{i}^{3})^{1/3},\quad\lim_{t\rightarrow-\infty}z(t)=-2-\beta_{\infty},\label{eq:z-infinity}
\end{equation}

\begin{equation}
\lim_{t\rightarrow+\infty}C(t)=\lim_{t\rightarrow+\infty}\sum_{i=1}^{q-1}e^{4t(1+\frac{1}{\beta_{i}})}=0,\quad\lim_{t\rightarrow-\infty}C(t)=0.\label{eq:C-infinity}
\end{equation}

\begin{proof}
For sufficiently large $t$, $\{u>t\}$ contains $q-1$ connected
components, which belong to small balls $B_{r}(p_{i})$ for $i=1,\cdots,q-1$
for some small $r.$

Denote 
\[
A_{i}(t)=\fint_{\{u>t\}\cap B_{r}(p_{i})}e^{4u}dl
\]

and 
\[
D_{i}(t)=\frac{1}{4}\fint_{\{u=t\}\cap B_{r}(p_{i})}2H|\nabla u|^{2}-2|\nabla u|^{3}dl,
\]

where $A(t)=\sum_{i=1}^{q-1}A_{i}(t)$ and $D(t)=\sum_{i=1}^{q-1}D_{i}(t).$

By making use of the equation (\ref{k-equation}),

\begin{align*}
A_{i}(t) & =\frac{2}{3}(D_{i}(t)-\lim_{r_{0}\rightarrow0}\frac{1}{4}\fint_{\partial B_{r_{0}}(p_{i})}2H|\nabla u|^{2}-3|\nabla u|^{3}+|\nabla u|^{3}dl)\\
 & =\frac{2}{3}(D_{i}(t)-(\frac{3}{2}|\beta_{i}|^{2}-\frac{1}{2}|\beta_{i}|^{3})).
\end{align*}
\\
By $\lim_{t\rightarrow\infty}A_{i}(t)=0,$ 
\[
\lim_{t\rightarrow\infty}D_{i}(t)=\frac{3}{2}|\beta_{i}|^{2}-\frac{1}{2}|\beta_{i}|^{3}.
\]

As
\begin{align*}
D_{i}(t) & =\frac{1}{4}\fint_{\{u=t\}\cap B_{r}(p_{i})}2H|\nabla u|^{2}-2|\nabla u|^{3}dl\\
 & =\frac{1}{4}\fint_{\beta_{i}\log|x-p_{i}|+v(x)=t}2(\frac{3}{|x-p_{i}|}+o(\frac{1}{|x-p_{i}|}))(\frac{\beta_{i}^{2}}{|x-p_{i}|^{2}}+o(\frac{1}{|x-p_{i}|^{2}})|dl\\
 & \quad-\frac{1}{4}\fint_{\beta_{i}\log|x-p_{i}|+v(x)=t}2(\frac{|\beta_{i}|^{3}}{|x-p_{i}|^{3}}+o(\frac{1}{|x-p_{i}|^{3}})|dl\\
 & =\frac{1}{4}\fint_{\beta_{i}\log|x-p_{i}|+v(x)=t}(\frac{6\beta_{i}^{2}-2|\beta_{i}|^{3}}{|x-p_{i}|^{3}}+o(\frac{1}{|x-p_{i}|^{3}}))\\
 & =(\frac{3}{2}|\beta_{i}|^{2}-\frac{1}{2}|\beta_{i}|^{3})\fint_{\beta_{i}\log|x-p_{i}|+v(x)=t}(\frac{1}{|x-p_{i}|^{3}}+o(\frac{1}{|x-p_{i}|^{3}})),
\end{align*}

taking $t\rightarrow\infty,$ we have
\[
\fint_{\beta_{i}\log|x-p_{i}|+v(x)=t}(\frac{1}{|x-p_{i}|^{3}}+o(\frac{1}{|x-p_{i}|^{3}}))dl\longrightarrow1.
\]

With the above equality, we get the other equalities.

Therefore we have proved the lemma.
\end{proof}
For future use, we prove the following
\begin{lem}
\label{lem:fundamental relation}Notations as above, we have
\begin{align}
C' & =A'+4C,\label{eq:C-A relation}\\
\frac{1}{3}\frac{d}{dt}(z^{3}) & =\fint_{L(t)}(\frac{H}{3}|\nabla u|-\nabla_{44}u)|\nabla u|.\label{eq:derivative of energy}
\end{align}
\end{lem}

\begin{proof}
As $A(t)$ and $B(t)$ are non-increasing with respect to $t,$ $A'(t)$
and $B'(t)$ exist almost everywhere. From the co-area formula (see
Lemma 2.3 in \cite{BZ}), we have, for given $t_{1}<t_{2},$
\[
B(t_{1})-B(t_{2})=\mathcal{H}^{4}(S\cap u^{-1}(([t_{1},t_{2}))+\int_{t_{1}}^{t_{2}}\int_{L(s)\backslash S}|\nabla u|^{-1}d\mathcal{H}^{3}ds,
\]

By Lemma \ref{lem:measure 0}, $\mathcal{H}^{3}(S\cap u^{-1}([t_{1},t_{2}))=0.$
Thus, $B(t)$ is absolutely continuous in any finite interval $[t_{1},t_{2}]$
as in \cite{FL2}. Similarly, $A(t)$, $C(t)$ and $z(t)$ are absolutely
continuous. Therefore, when these derivatives exist, we have:

\begin{align*}
B^{'}(t) & =\frac{1}{|S_{3}|}\frac{\partial}{\partial t}\int_{t}^{+\infty}\int_{L(s)\backslash S}\frac{1}{|\nabla u|}dlds,\\
 & =-\frac{1}{|S_{3}|}\int_{L(t)\backslash S}\frac{1}{|\nabla u|}dl,
\end{align*}
and 
\begin{align*}
A'(t) & =\frac{1}{|S_{3}|}\frac{\partial}{\partial t}\int_{t}^{+\infty}\int_{L(s)\backslash S}\frac{e^{4u}}{|\nabla u|}dlds\\
 & =e^{4t}B'(t).
\end{align*}
Note that by Lemma \ref{lem:measure 0}, $\mathcal{H}^{3}(S\cap L(t))=0.$
From here on, when no confusion arises, integrals over $L(t)$ should
be thought as integrals over $L(t)\backslash S$. 
\begin{align*}
C'(t) & =4e^{4t}B(t)+e^{4t}B'(t)\\
 & =4C(t)+A'(t).
\end{align*}

By the divergence theorem, co-area formula, and (\ref{eq:gauss-weingarten1}),
we obtain

\begin{align*}
\frac{d}{dt}\int_{L(t)}|\nabla u|^{3} & =\frac{d}{dt}\int_{S(t)\backslash S(t_{0})}-{\rm div}(\nabla u|\nabla u|^{2})\\
 & =\int_{L(t)}\frac{{\rm div}(\nabla u|\nabla u|^{2})}{|\nabla u|}
\end{align*}
where $t_{0}$ is a fixed real number. Note that for any $P\in L(t)\backslash S$,
using our preferred coordinate system, we have $u_{4}(P)=-|\nabla u|(P)$
and

\begin{align*}
\frac{{\rm div}(\nabla u|\nabla u|^{2})}{|\nabla u|} & =\frac{\sum_{\alpha=1}^{3}\nabla_{\alpha\alpha}u|\nabla u|^{2}}{|\nabla u|}-\frac{\sum_{\alpha=1}^{3}\nabla_{\alpha}u\nabla_{\alpha}|\nabla u|^{2}}{|\nabla u|}+\frac{\nabla_{44}u|\nabla u|^{2}}{|\nabla u|}+\frac{\nabla_{4}u\nabla_{4}|\nabla u|^{2}}{|\nabla u|}\\
 & =-H|\nabla u|^{2}+3\nabla_{44}u|\nabla u|.
\end{align*}
We thus get (\ref{eq:derivative of energy}).

\end{proof}
Finally, we will use the divergence structure of $\sigma_{m}$ and
equation (\ref{k-equation}) to establish the relationship between
$A$ and $D.$ 
\begin{thm}
\label{lem:D-Gauss}Let $u(x)$ be a smooth solution to (\ref{k-equation})
on $\mathbb{R}^{4}\backslash\{p_{1},\,\cdots,\,p_{q-1},\,p_{\infty}\}$
with $e^{2u}g_{E}\in\mathcal{C}_{2}^{+}$ , 
\begin{equation}
\frac{1}{|S_{3}|}\int_{\{u\ge t\}}\sigma_{2}(A)dvol=D(t)-D(+\infty),\label{eq:D-Gauss}
\end{equation}

\begin{equation}
A(t)=\frac{2}{3}(D(t)-D(+\infty)),\label{eq:A-D-relation}
\end{equation}

\begin{equation}
A(t)=\frac{1}{6}(\Sigma_{0}+\Sigma_{1})-\frac{2}{3}D(+\infty),\label{eq:A-=00003D00005Csigmarelation}
\end{equation}

\begin{equation}
A'(t)=\frac{2}{3}D'(t)=\frac{1}{6}(\Sigma_{0}'+\Sigma_{1}').\label{eq:A'-=00003D00005Csigma' relation}
\end{equation}
\end{thm}

\begin{proof}
By Lemma \ref{lem:divform} and (\ref{eq:expression of A}), 
\begin{align*}
\int_{\{u\ge t\}}2\sigma_{2}(g^{-1}A_{g})dv_{g} & =\int_{\{u=t\}}H|\nabla u|^{2}-\int_{\{u=t\}}|\nabla u|^{3}-2D(+\infty).
\end{align*}
So we can get (\ref{eq:D-Gauss}-\ref{eq:A'-=00003D00005Csigma' relation})
by equation (\ref{k-equation}).
\end{proof}

\section{Proof of main theorems}

In this section, we prove our main theorems.

First, we state our key estimate.
\begin{lem}
\label{lem:key inequality } Let u(x) be a smooth solution to equation
(\ref{k-equation}) on $\mathbb{R}^{4}\backslash\{p_{1},\,\cdots,\,p_{q}\}$
with $e^{2u}g_{E}\in\mathcal{C}_{2}^{+}.$ We have

\[
z'\fint_{L(t)}\sigma_{1}(\widetilde{A})|\nabla u|dl(zA')^{2}\ge\frac{3}{2}(4C(t))^{3}.
\]
\end{lem}

\begin{proof}
By (\ref{eq:expression of A}), we obtain 
\begin{align}
\sigma_{2}(A) & =\sigma_{2}(\widetilde{A})+(-\nabla_{44}u+\frac{|\nabla u|^{2}}{2})\sigma_{1}(\widetilde{A})-\sum_{a=1}^{3}(\nabla_{a4}u)^{2}\nonumber \\
 & \le\sigma_{2}(\widetilde{A})+(-\nabla_{44}u+\frac{|\nabla u|^{2}}{2})\sigma_{1}(\widetilde{A}).\label{eq:key1}
\end{align}
Here $n=4.$

Especially, because $(\sum_{i=1}^{3}\lambda_{i})^{2}=\sum_{i=1}^{3}\lambda_{i}^{2}+2\sum_{i<j}\lambda_{i}\lambda_{j}\ge3\sigma_{2}$
holds, we have 
\begin{equation}
\sigma_{2}(\widetilde{A})\le\frac{\sigma_{1}^{2}(\widetilde{A})}{3},\label{eq:key2}
\end{equation}
for any $\widetilde{A.}$

Then by (\ref{eq:key1}) and (\ref{eq:key2}), 
\begin{align*}
\sigma_{2}(A) & \le\frac{\sigma_{1}(\widetilde{A})\sigma_{1}(\widetilde{A})}{3}+(-\nabla_{44}u+\frac{|\nabla u|^{2}}{2})\sigma_{1}(\widetilde{A})\\
 & \le\sigma_{1}(\widetilde{A})(\frac{H}{3}|\nabla u|-\nabla_{44}u).
\end{align*}

As we have 
\[
\sigma_{2}(A)=\frac{3}{2}e^{4u},
\]

we can then derive the following using Cauchy inequality

\begin{align}
 & \fint_{L(t)}\sigma_{1}(\widetilde{A})|\nabla u|dl\fint_{L(t)}(\frac{H}{3}|\nabla u|-\nabla_{44}u)|\nabla u|\nonumber \\
\ge & \bigg(\fint_{L(t)}\sqrt{\sigma_{1}(\widetilde{A})|\nabla u|(\frac{H}{3}|\nabla u|-\nabla_{44}u)|\nabla u|}\bigg)^{2}\nonumber \\
\ge & \bigg(\fint_{L(t)}\sqrt{\frac{3}{2}e^{4u}|\nabla u|^{2}}\bigg)^{2}\nonumber \\
\ge & \frac{3}{2}e^{4t}\big(\fint_{L(t)}|\nabla u|\big)^{2}.\label{eq:key 3}
\end{align}

Combine equality (\ref{eq:key 3}) with Lemma \ref{lem:fundamental relation},
we get that 
\begin{align*}
 & (A')^{2}\fint_{L(t)}\sigma_{1}(\widetilde{A})|\nabla u|\fint_{L(t)}(\frac{H}{3}|\nabla u|-\nabla_{44}u)|\nabla u|dl\\
\ge & \frac{3}{2}\big(\fint_{L(t)}|\nabla u|\big)^{2}e^{4t}\cdot e^{8t}(\fint_{L(t)}\frac{1}{|\nabla u|})^{2}\\
= & \frac{3}{2}e^{12t}(\fint_{L(t)}|\nabla u|)^{2}(\fint_{L(t)}\frac{1}{|\nabla u|})^{2}\\
\ge & \frac{3}{2}e^{12t}|L(t)|^{4}\frac{1}{|S_{3}|^{4}}\\
\ge & \frac{3}{2}e^{12t}B(t)^{3}4^{4}|B_{1}|\frac{1}{|S_{3}|^{4}}\\
= & \frac{3}{2}(4C(t))^{3},
\end{align*}
where the third inequality is due to Cauchy inequality and the fourth
inequality holds because of the isoperimetric inequality. We note
that $|B_{1}|$=$\frac{|S_{3}|}{4}$. 
\end{proof}
We remark that if $u$ is the radial solution to equation (\ref{k-equation}),
all the above inequalities are equality. 

We then prove the following 
\begin{thm}
\textup{\label{thm:monotone inequality}} Let u(x) be a smooth solution
to (\ref{k-equation}) on $\mathbb{R}^{4}\backslash\{p_{1},\,\cdots,\,p_{q-1},\,p_{\infty}\}$
with $e^{2u}g_{E}\in\mathcal{C}_{2}^{+}$,\textup{ }then for generic
$t$,\textup{
\[
C'\le\frac{2}{3}D'+\frac{4}{9}(zD)'+\frac{1}{36}(z^{4})'.
\]
 In particular, the quantity 
\[
M(t)=\frac{2}{3}D(t)+\frac{4}{9}D(t)z(t)+\frac{1}{36}z^{4}(t)-C(t)
\]
 is monotonously increasing with respect to $t.$} 
\end{thm}

\begin{proof}
By Lemma \ref{lem:key inequality } and the inequality of arithmetic
and geometric means, we have

\begin{align}
4C & \le\frac{1}{3}(2zA'+\frac{2}{3}z'\fint_{L(t)}\sigma_{1}(\widetilde{A})|\nabla u|),\label{eq:11}
\end{align}
where the identity holds if and only if $zA'=\frac{2}{3}z'\fint_{L(t)}\sigma_{1}(\widetilde{A})|\nabla u|.$

Using fact that $C'=4C+A'$, (\ref{eq:11}) then becomes

\begin{align*}
C' & \le A'+\frac{1}{3}(2zA'+\frac{2}{3}z'\fint_{L(t)}\sigma_{1}(\widetilde{A})|\nabla u|)\\
 & \le\frac{1}{3}(2D'+\frac{4}{3}z(D'(t))+\frac{1}{3}z'\Sigma_{1})\\
 & =\frac{2}{3}D'+\frac{1}{3}\big(\frac{4}{3}(zD)'-\frac{4}{3}z'D+\frac{1}{3}z'\Sigma_{1}\big)\\
 & =\frac{2}{3}D'+\frac{1}{3}\big(\frac{4}{3}(zD)'-\frac{1}{3}z'(\Sigma_{0}+\Sigma_{1})+\frac{1}{3}z'\Sigma_{1}\big)\\
 & =\frac{2}{3}D'+\frac{1}{3}\big(\frac{4}{3}(zD)'+\frac{1}{3}z'z^{3}\big).
\end{align*}

The key estimate is thus established.
\end{proof}
Finally, we are ready to prove our main result, Theorem \ref{thm:Main theorem-1}.
\begin{proof}
By Theorem \ref{thm:monotone inequality}, we have

\begin{align*}
 & C(+\infty)-C(-\infty)\\
\le & \frac{2}{3}(D(+\infty)-D(-\infty))+\frac{4}{9}(z(+\infty)D(+\infty)-z(-\infty)D(-\infty))+\frac{1}{36}(z^{4}(+\infty)-z^{4}(-\infty)).
\end{align*}
Considering (\ref{eq:C-infinity}), (\ref{eq:D+infinity}), (\ref{eq:D-infinity})
and (\ref{eq:z-infinity}), we compute to get

\[
\frac{3}{8}\beta_{\infty}^{2}(\beta_{\infty}+2)^{2}\le\frac{3}{8}\widetilde{\beta}^{2}(\widetilde{\beta}+2)^{2}+(\widetilde{\beta}+\frac{3}{2})(\sum_{i=1}^{q-1}\beta_{i}^{2}-\widetilde{\beta}^{2}).
\]
Therefore, if $\frac{3}{8}\beta_{\infty}^{2}(\beta_{\infty}+2)^{2}>\frac{3}{8}\widetilde{\beta}^{2}(\widetilde{\beta}+2)^{2}+(\widetilde{\beta}+\frac{3}{2})(\sum_{i=1}^{q-1}\beta_{i}^{2}-\widetilde{\beta}^{2})$,
such solutions do not exist.

Furthermore, when $\frac{3}{8}\beta_{\infty}^{2}(\beta_{\infty}+2)^{2}=\frac{3}{8}\widetilde{\beta}^{2}(\widetilde{\beta}+2)^{2}+(\widetilde{\beta}+\frac{3}{2})(\sum_{i=1}^{q-1}\beta_{i}^{2}-\widetilde{\beta}^{2})$,
all the inequalities in Lemma \ref{lem:key inequality } and \ref{thm:monotone inequality}
become equalities. In particular, the isoperimetric inequality being
sharp implies that $u(x)$ is radial, which means that $u$ has at
most two singular points $p_{1},\,p_{2}=\infty.$ It is easy to check
that $\beta_{1}=\beta_{2}.$ We have proved that $M$ is the football
described first by \cite{CHY}.
\end{proof}
\begin{rem}
Li in \cite{L} used the moving sphere method to prove that a smooth
solution $u$ to equation (\ref{k-equation}) on $\mathbb{R}^{4}\backslash\{0\}$
in $\mathcal{C}_{2}^{+}$ is radial. Combine with Chang-Han-Yang's
\cite{CHY}, the radial solution to equation (\ref{k-equation}) is
a football. This is a generalization of the uniqueness result for
the smooth case, which is first proved in \cite{V1}. Our approach
is different.
\end{rem}

\end{document}